\documentclass{eptcs}

\usepackage{amsthm}
\usepackage{mathtools}

\usepackage[inline]{enumitem}
\usepackage{ifthen}
\usepackage[utf8]{inputenc} 
\usepackage{xcolor}

\usepackage[backend=biber, backref=true, maxbibnames = 10, style = alphabetic]{biblatex}

\usepackage{tikz}

\usepackage{amssymb}
\usepackage{mathrsfs}
\usepackage{dutchcal}
\usepackage{fontawesome}
\usepackage{ebproof}
\usepackage{stmaryrd}

  \hypersetup{final}

  \setlist{nosep}
  \setlistdepth{6}

  \usetikzlibrary{ 
  	cd,
  	math,
  	decorations.markings,
		decorations.pathreplacing,
  	positioning,
  	arrows.meta,
  	shapes,
		shadows,
		shadings,
  	calc,
  	fit,
  	quotes,
  	intersections,
    circuits,
    circuits.ee.IEC
  }
  
  \tikzset{
biml/.tip={Glyph[glyph math command=triangleleft, glyph length=.95ex]},
bimr/.tip={Glyph[glyph math command=triangleright, glyph length=.95ex]},
}

\tikzset{
	tick/.style={postaction={
  	decorate,
    decoration={markings, mark=at position 0.5 with
    	{\draw[-] (0,.4ex) -- (0,-.4ex);}}}
  }
} 
\tikzset{
	slash/.style={postaction={
  	decorate,
    decoration={markings, mark=at position 0.5 with
    	{\draw[-] (.3ex,.3ex) -- (-.3ex,-.3ex);}}}
  }
}

\theoremstyle{definition}
\newtheorem{definitionx}{Definition}[section]
\newtheorem{examplex}[definitionx]{Example}
\newtheorem{remarkx}[definitionx]{Remark}

\theoremstyle{plain}

\newtheorem{theorem}[definitionx]{Theorem}
\newtheorem{proposition}[definitionx]{Proposition}
\newtheorem{corollary}[definitionx]{Corollary}
\newtheorem{lemma}[definitionx]{Lemma}

\newtheorem*{theorem*}{Theorem}
\newtheorem*{proposition*}{Proposition}
\newtheorem*{corollary*}{Corollary}
\newtheorem*{lemma*}{Lemma}
\newtheorem*{warning*}{Warning}

\newenvironment{example}
  {\pushQED{\qed}\examplex}
  {\popQED\endexamplex}
  
 \newenvironment{remark}
  {\pushQED{\qed}\remarkx}
  {\popQED\endremarkx}
  
  \newenvironment{definition}
  {\pushQED{\qed}\definitionx}
  {\popQED\enddefinitionx}

\DeclareSymbolFont{stmry}{U}{stmry}{m}{n}
\DeclareMathSymbol\fatsemi\mathop{stmry}{"23}

\DeclareFontFamily{U}{mathx}{\hyphenchar\font45}
\DeclareFontShape{U}{mathx}{m}{n}{
      <5> <6> <7> <8> <9> <10>
      <10.95> <12> <14.4> <17.28> <20.74> <24.88>
      mathx10
      }{}
\DeclareSymbolFont{mathx}{U}{mathx}{m}{n}
\DeclareFontSubstitution{U}{mathx}{m}{n}
\DeclareMathAccent{\widecheck}{0}{mathx}{"71}

\renewcommand{\ss}{\subseteq}

\DeclareMathOperator{\ob}{Ob}

\newcommand{\Set}[1]{\text{#1}}
\newcommand{\Cat}[1]{\textbf{#1}}

\newcommand{\id}{\mathrm{id}}

\newcommand{\To}[2][]{\xrightarrow[#1]{#2}}

\newcommand{\card}{\,^{\#}}

\newcommand{\op}{^\tn{op}}

\newcommand{\tn}[1]{\textnormal{#1}}

\newcommand{\rr}{\mathbb{R}}

\newcommand{\smset}{\Cat{Set}}

\newcommand{\yon}{\mathcal{y}}
\newcommand{\poly}{\Cat{Poly}}
\newcommand{\dir}{\Set{Dir}}
\newcommand{\rect}{\Set{Rect}}
\newcommand{\polycart}{\poly^{\Cat{Cart}}}
\newcommand{\hh}{\mathcal{h}}

\newcommand{\tri}{\mathbin{\triangleleft}}
\newcommand{\R}{R}
\newcommand{\T}{T}

\newcommand{\biglens}[2]{
     \begin{bmatrix}{\vphantom{f_f^f}#2} \\ {\vphantom{f_f^f}#1} \end{bmatrix}
}
\newcommand{\littlelens}[2]{
     \begin{bsmallmatrix}{\vphantom{f}#2} \\ {\vphantom{f}#1} \end{bsmallmatrix}
}
\newcommand{\lens}[2]{
  \relax\if@display
     \biglens{#1}{#2}
  \else
     \littlelens{#1}{#2}
  \fi
}

\newcommand{\qand}{\quad\text{and}\quad}
\newcommand{\qqand}{\qquad\text{and}\qquad}

\allowdisplaybreaks
\begin{document}

\title{Polynomial Functors and Shannon Entropy}
\def\titlerunning{Polynomial Functors and Shannon Entropy}
\def\authorrunning{David I.\ Spivak}
\author{David I.\ Spivak}

\date{\vspace{-.2in}}

\maketitle

\begin{abstract}
Past work shows that one can associate a notion of Shannon entropy to a Dirichlet polynomial, regarded as an empirical distribution. Indeed, entropy can be extracted from any $d\in\dir$ by a two-step process, where the first step is a rig homomorphism out of $\dir$, the \emph{set} of Dirichlet polynomials, with rig structure given by standard addition and multiplication. In this short note, we show that this rig homomorphism can be upgraded to a rig \emph{functor}, when we replace the set of Dirichlet polynomials by the \emph{category} of ordinary (Cartesian) polynomials.

In the Cartesian case, the process has three steps. The first step is a rig functor $\polycart\to\poly$ sending a polynomial $p$ to $\dot{p}\yon$, where $\dot{p}$ is the derivative of $p$. The second is a rig functor $\poly\to\smset\times\smset\op$, sending a polynomial $q$ to the pair $(q(1),\Gamma(q))$, where $\Gamma(q)=\poly(q,\yon)$ can be interpreted as the global sections of $q$ viewed as a bundle, and $q(1)$ as its base. To make this precise we define what appears to be a new distributive monoidal structure on $\smset\times\smset\op$, which can be understood geometrically in terms of rectangles. The last step, as for Dirichlet polynomials, is simply to extract the entropy as a real number from a pair of sets $(A,B)$; it is given by $\log A-\log \sqrt[A]{B}$ and can be thought of as the log aspect ratio of the rectangle.
\end{abstract}

\section{Introduction}

In practice, a probability distribution on a set of \emph{outcomes} arises from considering finite samples. A sample consists of a set of observations, or \emph{draws}, each corresponding to one of the outcomes. For example, the following is a sample with five (5) outcomes and eight (8) draws:

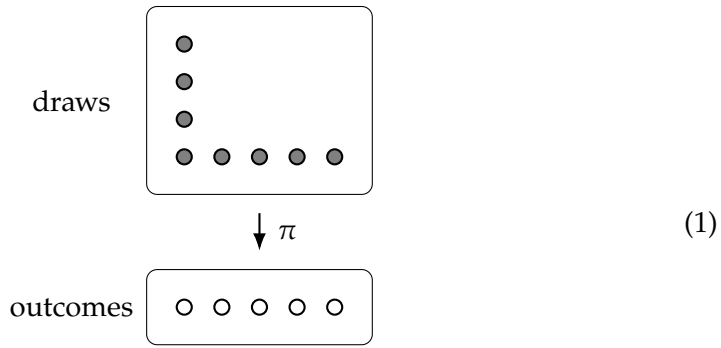
\begin{equation}\label{eqn.empirical}
  \begin{tikzpicture}[scale=0.5, baseline=(pi)]
    \node at (-2,2.5) {draws};
    \begin{scope}
      \draw[rounded corners] (0,0) rectangle ++(6,5);
      \foreach \y in {1,2,3,4}
        \draw[thick,black,fill=gray] (1,\y) circle (2mm);
      \foreach \x in {2,3,4,5}
        \draw[thick,black,fill=gray] (\x,1) circle (2mm);
    \end{scope}
    \draw[thick,-Latex] (3,-0.5) to (3,-1.5);
     \node at (3.75,-1) (pi) {$\pi$};
   \node at (-2,-3) {outcomes};
    \begin{scope}[shift={(0,-4)}]
      \draw[rounded corners] (0,0) rectangle ++(6,2);
      \foreach \x in {1,2,3,4,5}
        \draw[thick,black,fill=white] (\x,1) circle (2mm);
    \end{scope}
  \end{tikzpicture}
\end{equation}

This corresponds to the probability distribution $P=(\frac{1}{2},\frac{1}{8},\frac{1}{8},\frac{1}{8},\frac{1}{8})$. But the sample itself can be encoded in the form of a polynomial, namely $p\coloneqq\yon^4+4\yon$. Note that $p(1)=5$ is the number of outcomes and that $\dot{p}(1)=8$ is the number of draws, where $\dot{p}=4\yon^3+4$ is the derivative of $p$. The map $\pi$ itself is somehow inherent in $p$: one of its summands has an exponent of $4$, whereas its other four summands each have an exponent of $1$. Yet one may wonder: is this polynomial encoding really meaningful, or is it just a bizarre packaging of the sample? Our goal in this paper is to show that it is meaningful, at least when it comes to considering the Shannon entropy $H(P)$. 

The Shannon entropy of a distribution \cite{shannon1948mathematical} is a measure of how much information is transmitted when outcomes are selected according to the distribution. For example, if one repeatedly chooses an element of the 8 draws in diagram \eqref{eqn.empirical} uniformly at random but only reports the outcome, then the first outcome will show up four-times more often than any other. As we will explain, Shannon's information theory says that this distribution has entropy $H(P)=2$, i.e.\ it transmits the same amount of information as if it were a uniform distribution on only $4$ outcomes.

In this paper we will give a category-theoretic account of the Shannon entropy of the probability distribution corresponding to a sample encoded as a polynomial $p$, or more precisely a \emph{polynomial functor} $p\in\poly$. Polynomial functors are ubiquitous: they show up in type theory \cite{avigad2019data,awodey2018polynomial}, dynamical systems theory \cite{spivak2020poly,spivak2022poly}, database theory \cite{spivak2015relational,spivak2021functorial}, programming language theory \cite{bird1996algebra,abbott2003categories}, and higher category theory \cite{thanh2019sequent,shapiro2021familial}. 

The category $\poly$ of polynomial functors in one variable has an enormous amount of structure. For example, it has at least eight distinct monoidal structures, of which two will be relevant to us. One is the coproduct: given two polynomials $p,q$, we may add them to form $p+q$. In terms of samples, this operation simply takes the disjoint union of two samples: both the sets of outcomes and the sets of draws. The other is the \emph{Dirichlet product}, denoted $\otimes$. We will give the precise formula for $p\otimes q$ in Section~\ref{sec.rig}, but the idea is that it runs the two samples independently: an outcome in $p\otimes q$ is a pair consisting of an outcome from $p$ and an outcome from $q$, and a draw is also a pair consisting of a draw from $p$ and a draw from $q$.

These two operations make $\poly$ a \emph{distributive monoidal category}, because $\otimes$ distributes over $+$. The goal of this paper is to show that most of the process for taking the Shannon entropy of a sample is fully category-theoretic. Indeed, we will factor the process into three stages, the first two of which are completely categorical, and the last of which extracts a real number that will be the entropy. 

The first stage is to define a rig functor $\T\colon\polycart\to\poly$, which sends a polynomial $p$ to $\T(p)\coloneqq\dot{p}\yon$, where $\dot{p}$ is the derivative of $p$. The second stage is to define a rig functor $\R\colon\poly\to\smset\times\smset\op$, which sends a polynomial $q$ to $\R(q)\coloneqq(q(1),\Gamma(q))$, where $\Gamma(p)=\poly(p,\yon)$ can be construed as the set of global sections of $p$, viewed as a bundle. 

The fact that both $\T$ and $\R$ are rig functors means that each preserves both the coproduct and the $\otimes$-product, a surprising amount of structure. But to say this, we need to define what appears to be a novel symmetric monoidal product $\otimes$ on $\smset\times\smset\op$. It is given by
\[
	(A_1,B_1)\otimes(A_2,B_2)\coloneqq\left(A_1A_2\,,\,B_1^{A_2}B_2^{A_1}\right).
\]
This monoidal product $\otimes$ distributes over the coproduct, which is given by
\[
(A_1,B_1)+(A_2,B_2)\coloneqq\left(A_1+A_2\,,\,B_1\times B_2\right),
\]
hence making $\smset\times\smset\op$ a distributive monoidal category, and in particular a \emph{rig category}. We will explain these two rig functors $\T$ and $\R$ in Section~\ref{chap.main}. We denote their composite---the result of the first and second stages---by
\[\hh\coloneqq(\R\circ\T)\colon\polycart\to\smset\times\smset\op.\]
It contains the categorical aspect of the entropy in a given sample $p\in\polycart$.

Before we discuss the third stage, we need a bit of intuition. Namely, we can think of an object $(A,B)\in\smset\times\smset\op$ as encoding a rectangle that has length $A$ and width $\sqrt[A]{B}$. The coproduct of two rectangles is given by adding their lengths and taking the geometric mean of their widths. The $\otimes$-product of two rectangles is given by multiplying both their lengths and their widths. It is in these terms that we can understand the third and final stage, which is simply to take the \emph{log aspect ratio} (the log of the quotient of length divided by width) of a given rectangle:
\[L(A,B)=\log A-\log\sqrt[A]{B}.\]

That is, we will prove that for any polynomial $p$ with an associated probability distribution $P$, the Shannon entropy $H(P)$ can be computed by first applying the rig-functorial operation to obtain $\hh(p)\in\smset\times\smset\op$, and then by extracting the log aspect ratio:
\[
H(P)=L(\hh(p)).
\]

We will conclude by returning to our original example, after giving the full composite: the function that takes a polynomial $p$ and returns the entropy of the corresponding empirical distribution is given by
\[
  L(\hh(p))\coloneqq\log \dot{p}(1)-\frac{\log\Gamma(\dot{p}\yon)}{\dot{p}(1)}
\]
Note that $\log\sqrt[A]{B}=\frac{\log B}{A}$.

So consider again the polynomial $p=\yon^4+4\yon$, depicted in \eqref{eqn.empirical}. Then we calculate
\[
\dot{p}\yon=4\yon^4+4\yon
,\quad
\dot{p}(1)=8
,\quad
\Gamma(\dot{p}\yon)=4^4*1^4=2^8
,\qand 
L(\hh(p))=\log 8-\frac{\log 2^8}{8}=2
\]
which agrees with our former calculation: its entropy is $H(P)=L(\hh(p))=2$. 

The remainder of this note is divided into two sections: Section~\ref{chap.background} gives background on polynomial functors, including the definition of $\polycart\ss\poly$ as well as the $+$ and $\otimes$ structures. Section~\ref{chap.main} gives the main results: explaining the seemingly novel distributive monoidal structure on $\smset\times\smset\op$, providing a rig monoidal functor $\hh\colon\polycart\to\smset\times\smset\op$, showing how to extract the entropy via a partial function $L\colon\ob(\smset\times\smset\op)\to\rr$, and finally proving the main theorem: that $H(P)=L(\hh(p))$.

There have been other categorical approaches to entropy, most notably \cite{baez2011characterization}, \cite{baez2014bayesian}, \cite{leinster2021entropy}, and \cite{parzygnat2022functorial}. Our presentation here has almost nothing in common with those. 

However, this work is closely aligned with \cite{spivak2021dirichlet}. There, the authors---myself and Tim Hosgood---use Dirichlet polynomials rather than ordinary (Cartesian)%
\footnote{Ren\'{e} Decartes at least invented the notation, e.g.\ $\yon^2+3\yon+2$, for polynomials; hence we refer to them as \emph{Cartesian polynomials} when we need to distinguish them from Dirichlet polynomials.}
polynomials. At the time, we seemed to have a choice of whether to use Dirichlet or Cartesian polynomials, and the Dirichlet route seemed cleaner and more intuitive for talking about the bundles. However, we were missing a few key ideas at the time. Whereas there we only factored out from $H$ a rig homomorphism (a function) $\dir\to\rect$ to a somewhat ad hoc rig we called $\rect$, the presentation here factors out from $H$ a rig \emph{functor} $\polycart\to\smset\times\smset\op$. Thus it is a significant categorical upgrade.

\subsection*{Acknowledgments}
Thanks to Valeria de Paiva for interesting conversations, and thanks to the referees of ACT2022 for suggestions, e.g.\ leading to Remark~\ref{rem.poly_prod_monoidal}.

This material is based upon work supported by the Air Force Office of Scientific Research under award number FA9550-20-1-0348.

\section{Background on polynomial functors}\label{chap.background}

Readers familiar with the rig category $(\poly,0,+,\yon,\otimes)$ should skip to Section~\ref{sec.CT}.

\subsection{Basics}

The main purpose of this section is to fix notation and provide a brief overview of polynomial functors in one variable. More extensive background material can be found in \cite{spivak2022poly} and \cite{kock2012polynomial}. 

\begin{definition}[Polynomial functor]\label{def.poly}
Given a set $S$, we denote the corresponding representable functor by
\[\yon^S\coloneqq\smset(S,-)\colon\smset\to\smset,\]
e.g. $\yon^S(X)\coloneqq X^S$. In particular $\yon=\yon^1$ is the identity and $\yon^0=1$ is constant singleton.

A \emph{polynomial functor} is a functor $p\colon\smset\to\smset$ that is isomorphic to a sum of representables, i.e.\ for which there exists a set $T$, a set $p[t]\in\smset$ for each $t\in T$, and an isomorphism of functors
\[
p\cong\sum_{t\in T}\yon^{p[t]}.
\]
We refer to $T$ as the set of \emph{$p$-types}, and for each type $t\in T$ we refer to $p[t]$ as the set of \emph{$p$-terms of type $t$}.%

A \emph{morphism} $\varphi\colon p\to p'$ of polynomial functors is simply a natural transformation between them. It is called \emph{cartesian} if for every map of sets $f\colon S\to S'$, the naturality square
\[
\begin{tikzcd}
  p(S)\ar[r, "p(f)"]\ar[d, "\varphi(S)"']&p(S)\ar[d, "\varphi(S')"]\\
  p'(S')\ar[r, "p'(f)"']&p'(S')\ar[ul, phantom, very near end, "\lrcorner"]
\end{tikzcd}
\]
is a pullback of sets. We denote the category of polynomial functors by $\poly$ and the wide subcategory of polynomials and cartesian maps by $\polycart\ss\poly$.\end{definition}

For any polynomial $p=\sum_{t\in T}\yon^{p[t]}$, we have a canonical isomorphism $p(1)\cong T$; hence from now on we will denote $p$ by
\begin{equation}\label{eqn.poly_notation}
p=\sum_{I\in p(1)}\yon^{p[I]}
\end{equation}
so that each $p$-types is written with an upper-case letter, e.g. $I\in p(1)$, and its terms are written with corresponding lower-case letters, e.g. $i\in p[I]$.

\begin{remark}\label{rem.positions_and_directions}
Using the Yoneda lemma, the fact that a morphism in $\poly$ is just a natural transformation, and the fact that a polynomial is a coproduct of representables, we derive
\begin{align*}
	\poly(p,q)&=
	\poly\left(\sum_{I\in p(1)}\yon^{p[I]},\sum_{J\in p(1)}\yon^{q[J]}\right)\\&\cong
	\prod_{I\in p(1)}\poly\left(\yon^{p[I]},\sum_{J\in p(1)}\yon^{q[J]}\right)\\&\cong
	\prod_{I\in p(1)}\sum_{J\in q(1)}\smset(q[J],p[I]).
\end{align*}
Thus we can understand a morphism $p\to q$ in $\poly$ to consist of two parts $(\varphi_1,\varphi^\sharp)$ as follows:
\begin{equation}\label{eqn.mapsharp}
  \varphi_1\colon p(1)\to q(1)
  \qqand
  \varphi^\sharp_I\colon q[J]\to p[I],
\end{equation}
where $J\coloneqq\varphi_1(I)$. That is, $\varphi_1$ is a function from $p$-types to $q$-types, and $\varphi^\sharp_i$ is a function on terms that \emph{depends on a choice of position $I\in p(1)$}. We refer to $\varphi_1$ as the \emph{on-types function} and to $\varphi^\sharp$ as the \emph{backwards on-terms} function. 

One can check that a map $\varphi\colon p\to q$ is cartesian iff the backwards-on-terms function $\varphi^\sharp_I$ is a bijection $p[I]\cong q[\varphi_1I]$ for each type $I\in p(1)$.
\end{remark}

\begin{example}[Types and global sections, $p(1)$ and $\Gamma(p)$]\label{ex.pos_glob}
For any polynomial $p$, we will be particularly interested in two sorts of maps: $\yon\to p$ and $p\to\yon$. The former is easy: a map $\yon\to p$ is given on types by choosing a single type $I\in p(1)$ to be the image of the unique type $!\in\yon(1)$ and given backward on terms using the unique choice of function $p[I]\to 1=\yon[!]$. Thus we have $p(1)\cong\poly(\yon,p)$.

More interesting are the maps $\gamma\colon p\to\yon$. This time $\gamma$ is trivial on types: each type $I\in p(1)$ is sent to the unique type $!\in\yon(1)$. However on terms, we need a map $\varphi^\sharp_I\colon 1\to p[I]$ for each $I$, meaning a choice of term $i\in p[I]$ for each $I\in p(1)$. In other words, writing $\Gamma(p)\coloneqq\poly(p,\yon)$, we have
\begin{equation}\label{eqn.global}
\Gamma(p)\cong\prod_{I\in p(1)}p[I].
\end{equation}
We refer to $\Gamma(p)$ as the set of \emph{global sections} of $p$, as is justified by the bundle terminology the next section.

Note that $-(1)\colon\poly\to\smset$ and $\Gamma\colon\poly\to\smset\op$ are functorial, as they are represented and corepresented by $\yon\in\poly$. We will be very interested in the functor
\begin{equation}\label{eqn.fundamental}
\R\colon\poly\to\smset\times\smset\op
\end{equation}
given by $\R(p)\coloneqq (p(1),\Gamma(p))$. In fact, $\R$ is a left adjoint, but we do not need that for this paper. In Remark~\ref{rem.geomean} we will explain that $\R(p)$ can be viewed as the \emph{rectangular aspect} of the polynomial $p$, hence the name $\R$.
\end{example}

\subsection{Derivatives and bundles}

We can understand polynomial functors in terms of bundles, using the derivative. For any polynomial $p$, its derivative $\dot{p}$ is defined as follows:
\begin{equation}\label{eqn.dotp}
\dot{p}\coloneqq\sum_{I\in p(1)}\sum_{i\in p[I]}\yon^{p[I]-\{i\}}
\end{equation}
where $p[I]-\{i\}$ denotes the set-difference. Note that $\dot{p}(1)\cong\sum_{I\in p(1)}p[I]$ is the set of all $p$-terms, and it comes with a map $\dot{p}(1)\to p(1)$ to the set of $p$-types. Often in the literature, this map of sets---which we call a bundle---is taken to be the polynomial itself. A map of polynomials $\varphi\colon p\to q$ can be written in terms of these bundles:
\[
\begin{tikzcd}
	\dot{p}(1)\ar[d]&p(1)\times_{q(1)}\dot{q}(1)\ar[l, "\varphi^\sharp"']\ar[r]\ar[d]&\dot{q}(1)\ar[d]\\
	p(1)\ar[r, equal]&p(1)\ar[r, "\varphi_1"']&q(1)\ar[ul, phantom, very near end, "\lrcorner"]
\end{tikzcd}
\]
Just as in Remark~\ref{rem.positions_and_directions}, one provides a forward map on types $\varphi_1\colon p(1)\to q(1)$, at which point one takes the pullback of that map with $\dot{q}(1)\to q(1)$, and then one provides a backward map $\varphi^\sharp\colon p(1)\times_{q(1)}\dot{q}(1)\to \dot{p}(1)$ on directions.
Again, $\varphi$ is cartesian iff $\varphi^\sharp$ is a bijection.

We write $pq=p\times q$ for the usual product of two polynomials, e.g.\ $\dot{p}\yon=\dot{p}\times\yon$.

\begin{proposition}\label{prop.dotpy}
The assignment $p\mapsto\dot{p}\yon$ is a functor $\polycart\to\polycart$.
\end{proposition}
\begin{proof}
We can think of $\dot{p}\yon$ as follows:
\begin{equation}\label{eqn.derivy}
\dot{p}\yon\cong\sum_{I\in p(1)}\sum_{i\in p[I]}\yon^{p[I]}
\end{equation}
Given a cartesian map $\varphi\colon p\to q$, the bijection $\varphi^\sharp\colon q[\varphi_1(I)]\cong p[I]$ lets us define a map $\dot{p}\yon\to\dot{q}\yon$ in an obvious way. 
\end{proof}

\begin{remark}\label{rem.comonad}
In fact, the assignment $(p\mapsto\dot{p}\yon)\colon\polycart\to\polycart$ extends to a comonad on $\polycart$. The counit map $\epsilon_p\colon\dot{p}\yon\to p$ is cartesian and is given on types by $(I,i)\mapsto I$. The comultiplication $\delta_p\colon\dot{p}\yon\to\ddot{p}\yon^2+\dot{p}\yon$ is given by the coproduct inclusion.

A coalgebra for this comonad is a polynomial $p$ equipped with a map $\gamma\colon p\to\dot{p}\yon$ such that $\epsilon_p\circ\gamma=\id_p$; it is not hard to check that the other condition holds for free. Hence a coalgebra structure on $p$ can be identified with a choice a global section $p\to\yon$, i.e.\ an element $\gamma\in\Gamma(p)$. Of course the map $p\to\yon$ is not cartesian in general, so the only way it can be encoded in $\polycart$ is via this coalgebra structure. A map of coalgebras is a cartesian map $\varphi\colon p\to p'$ that commutes with the global sections: $\Gamma(\varphi)(\gamma')=\gamma$.

The above is intriguing in that $\Gamma(p)$ is a major player in the story of this paper, but we currently know of no further connection between entropy and this comonad. 
\end{remark}

\subsection{Rig monoidal structure on $\poly$}\label{sec.rig}

The category $\poly$ has coproducts $p+q$ and products $p\times q$ given by usual polynomial arithmetic. We will be more interested in the former:%
\footnote{
The only reason we introduce $\times$ for $\poly$ is to explain that the polynomial product $\dot{p}\yon$ is in fact the categorical product $\dot{p}\yon\cong\dot{p}\times\yon$.
}
coproducts constitute a symmetric monoidal product with unit $0$. A type in $p+q$ is a type in $p$ or disjointly a type in $q$, and a term of that type is as specified in $p$ or $q$, respectively.

We will also be interested in another monoidal product called \emph{Dirichlet product} and denoted $-\otimes-$; the types and terms of $p\otimes q$ are given by the following formula:
\begin{equation}\label{eqn.dir_formula}
  \left(\sum_{I\in p(1)}\yon^{p[I]}\right)\otimes
  \left(\sum_{J\in q(1)}\yon^{q[J]}\right)\coloneqq
  \sum_{(I,J)\in p(1)\times q(1)}\yon^{p[I]\times q[J]}.
\end{equation}
This gives a symmetric monoidal structure $(\poly,\yon,\otimes)$. A type in $p\otimes q$ is just a pair of types $(I,J)\in p(1)\times q(1)$ and a term of it is just a pair of terms $(i,j)\in p[I]\times q[J]$.

In the language of bundles, $p+q$ and $p\otimes q$ are respectively given by
\[
\begin{tikzcd}
	\dot{p}(1)+\dot{q}(1)\ar[d]&\dot{p}(1)\times\dot{q}(1)\ar[d]\\
	p(1)+q(1)&p(1)\times q(1)
\end{tikzcd}
\]
i.e. $\dot{(p+q)}(1)\cong\dot{p}(1)+\dot{q}(1)$ and $\dot{(p\otimes q)}(1)\cong\dot{p}(1)\times\dot{q}(1)$.

The $\otimes$-structure distributes over the $+$ structure:
\[
p\otimes (q_1+q_2)\cong (p\otimes q_1)+(p\otimes q_2),
\]
thus making $(\poly,0,+,\yon,\otimes)$ a distributive monoidal category, and in particular a rig monoidal category.

\begin{remark}[Leibniz and chain rules]\label{rem.leibniz}
Some readers may be interested in the Leibniz rule and chain rule, that
\begin{align*}
	\dot{(p\times q)}&\cong\dot{p}\times q+p\times\dot{q}\\
	\dot{(p\tri q)}&\cong(\dot{p}\tri q)\times\dot{q}
\end{align*}
where $\times$ is the categorical product and $\tri$ is the composition product in $\poly$. These hold, but we will not need them in this paper.
\end{remark}

\section{Main results}\label{chap.main}

We divide this section into two parts. Section~\ref{sec.CT} is the category theory part, in which we provide what seems to be a novel symmetric monoidal structure on $\smset\times\smset\op$ and show that both $p\mapsto\dot{p}\yon$ and $q\mapsto(q(1),\Gamma(q))$ are rig functors. At the end of this section, we will have a rig functor $\hh\colon\polycart\to\smset\times\smset\op$ that does the categorical work of Shannon entropy.

Section~\ref{sec.entropy} is the finishing step, providing a function $\ob(\smset\times\smset\op)\to\rr$ and showing that when it is combined with the above, the map $H\colon\ob(\polycart)\to\rr$ sends an appropriately finite polynomial $p$ to the Shannon entropy of the empirical distribution defined by $p$.

\subsection{Categorical entropy of a polynomial}\label{sec.CT}

Below we will often denote products of sets by juxtaposition, $AB\coloneqq A\times B$. Recall the functor $p\mapsto\dot{p}\yon$ from Proposition~\ref{prop.dotpy}.

\begin{proposition}
The functor $p\mapsto\dot{p}\yon$ is a rig functor $\polycart\to\polycart$. 
\end{proposition}
\begin{proof}
Clearly $\dot{0}=0$ and $\dot{(p+q)}\cong\dot{p}+\dot{q}$, and by multiplying both sides by $\yon$ we see that the functor $p\mapsto p\yon$ preserves the coproduct structure. There is an isomorphism $\dot{\yon}\yon\cong\yon$, and for any $p,q\in\polycart$ there is also an isomorphism $\dot{(p\otimes q)}\yon\cong(\dot{p}\yon)\otimes(\dot{q}\yon)$, as follows from \eqref{eqn.derivy} and \eqref{eqn.dir_formula}; thus $p\mapsto\dot{p}\yon$ preserves the $\otimes$-structure. All of these isomorphisms are natural in $p,q\in\polycart$, completing the proof.
\end{proof}

The following corollary is straightforward, since $\polycart$ inherits $+$ and $\otimes$ from the forgetful functor $\polycart\to\poly$.

\begin{corollary}\label{cor.D}
The functor $\T(p)\coloneqq\dot{p}\yon$ is a rig functor $\T\colon\polycart\to\poly$. 
\end{corollary}

\begin{remark}[Total polynomial]\label{rem.total}
Note that for any $p$ we have $(\dot{p}\yon)(1)\cong\dot{p}(1)$. We think of $\dot{p}\yon$ as the \emph{total polynomial} of $p$, akin to the total space of a bundle, where $p$ is playing the role of the base. To justify this intuition, note that $\dot{p}\yon$ comes with a ``projection'' map $\epsilon\colon\dot{p}\yon\to p$ and that a section $p\to \dot{p}\yon$ of $\epsilon$ can be identified with a section $\gamma\in\Gamma(p)$ of $p$ as a bundle; see Remark~\ref{rem.comonad,ex.pos_glob}.
\end{remark}

\begin{example}\label{ex.global_D}
For any polynomial $p$, we have
\[
\Gamma(\dot{p}\yon)\cong\prod_{I\in p(1)}p[I]^{p[I]}.
\]
This formula---which follows directly from Eq.~\ref{eqn.global,eqn.derivy}---will be relevant when connecting the category theory to Shannon entropy later on. 
\end{example}

\begin{proposition}
The category $\smset\times\smset\op$ has a distributive monoidal structure:
\begin{align}
  (A_1,B_1)+(A_2,B_2)&\coloneqq(A_1+A_2\,,\,B_1B_2)
  \label{eqn.sum_geomean}\\\label{eqn.prod_prod}
  (A_1,B_1)\otimes(A_2,B_2)&\coloneqq(A_1A_2\,,\,B_1^{A_2}B_2^{A_1})
\end{align}
The units are $(0,1)$ and $(1,1)$ respectively.
\end{proposition}
\begin{proof}
Coproducts in $\smset\op$ are products in $\smset$, justifying the first line; these clearly form a symmetric monoidal structure. For the $\otimes$-monoidal structure, note that the formula is functorial in $A\in\smset$ and $B\in\smset\op$. It is also symmetric as well as unital: $(1,1)\otimes(A_2,B_2)\cong(A_2,B_2)$. Associativity is justified as follows:
\begin{align*}
  (A_1,B_1)\otimes((A_2,B_2)\otimes(A_3,B_3))&\cong
	(A_1A_2A_3,B_1^{A_2A_3}B_2^{A_1A_3}B_3^{A_1A_2})\\&\cong
	((A_1,B_1)\otimes(A_2,B_2))\otimes(A_3,B_3).
\end{align*}
There is an absorption map $(0,1)\otimes (A,B)\cong(0,B)\to(0,1)$, and the distributivity of $\otimes$ over $+$ is justified as follows:
\begin{align*}
  (A,B)\otimes\big((A_1,B_1)+(A_2,B_2)\big)&\cong 
  \big(A(A_1+A_2),B^{A_1+A_2}(B_1B_2)^A\big)\\&\cong
  \big(AA_1+AA_2,B^{A_1}B^{A_2}B_1^AB_2^A\big)\\&\cong
  \big((A,B)\otimes(A_1,B_1))+((A,B)\otimes(A_2,B_2)\big).
\end{align*}
We leave the remaining details to the interested reader.
\end{proof}

\begin{remark}[Formal roots and rectangular aspect]\label{rem.geomean}
One can think of an object $(A,B)\in\smset\times\smset\op$ as formally representing the $A$th root of $B$, i.e.\ the number $\sqrt[A]{B}=B^{\frac{1}{A}}$, keeping track of the base $A$ as well. It is helpful to think of $(A,B)$ as a rectangle with length $A$ and width $\sqrt[A]{B}$. From this perspective, the sum from \eqref{eqn.sum_geomean} adds the lengths and takes the geometric mean of the widths, and the monoidal product from \eqref{eqn.prod_prod} takes the product of both lengths and widths:
\[
(B_1B_2)^\frac{1}{A_1+A_2}=\left(\left(\sqrt[A_1]{B_1}\right)^{A_1}\times \left(\sqrt[A_2]{B_2}\right)^{A_2}\right)^{\frac{1}{A_1+A_2}}
\qqand
(B_1^{A_2}B_2^{A_1})^\frac{1}{A_1A_2}=\sqrt[A_1]{B_1}\sqrt[A_2]{B_2}.
\]

For any polynomial $p$, the functor $\R(p)\coloneqq(p(1),\Gamma(p))$ from \eqref{eqn.fundamental} is consonant with this interpretation. We may say that $\R(p)$ denotes the \emph{rectangular aspect} of $p$ in the sense that $p(1)$ represents the length and $\sqrt[p(1)]{\Gamma(p)}$, the geometric mean of the fiber cardinalities, represents the width. For example, the polynomial $p=\yon^4+4\yon$, depicted in Diagram~\eqref{eqn.empirical}, has length $p(1)=5$ and width $\sqrt[5]{4}\approx1.3$.
\end{remark}

\begin{remark}
The $\otimes$ operation \eqref{eqn.prod_prod} on $\smset\times\smset\op$ in fact has a closure
\[
[(A_1,B_1),(A_2,B_2)]\coloneqq\left(A_2^{A_1}B_1^{B_2}\,,\,A_1B_2\right).
\]
We will not need this, but it is interesting that $\smset\times\smset\op$ has so much structure.
\end{remark}

\begin{proposition}\label{prop.W}
The functor $\R\colon\poly\to\smset\times\smset\op$ from \eqref{eqn.fundamental} is a rig functor.
\end{proposition}
\begin{proof}
Recall from \eqref{eqn.fundamental} that $\R(p)\coloneqq(p(1),\Gamma(p))$. Clearly $0(1)=0$ and $(p+q)(1)\cong p(1)+q(1)$. Also $\Gamma(0)=1$ and $\Gamma(p+q)\cong\Gamma(p)\times\Gamma(q)$; hence $\R$ preserves the $(0,+)$ monoidal structure. Moreover, we have $\yon(1)=1$ and $(p\otimes q)(1)\cong p(1)\times q(1)$ and $\Gamma(\yon)= 1$, so to show that $\R$ preserves the $(\yon,\otimes)$ monoidal structure, it remains only to provide an isomorphism 
\[
  \Gamma(p\otimes q)\cong \Gamma(p)^{q(1)}\times\Gamma(q)^{p(1)}.
\]
It is given as follows:
\begin{align*}
	\Gamma(p\otimes q)&\cong
	\prod_{(I,J)\in p(1)\times q(1)}p[I]q[J]\\&\cong
	\left(\prod_{(I,J)\in p(1)\times q(1)}p[I]\right)\times
		\left(\prod_{(I,J)\in p(1)\times q(1)}q[J]\right)\\&\cong
	\prod_{J\in q(1)}\prod_{I\in p(1)}p[I]\times\prod_{I\in p(1)}\prod_{J\in q(1)}q[J]\\&\cong
	\Gamma(p)^{q(1)}\times\Gamma(q)^{p(1)}	
\end{align*}
\end{proof}


We summarize the above section before we go on to the final one. Namely, the functors $\T\colon\polycart\to\poly$ and $\R\colon\poly\to\smset\times\smset\op$ from Corollary~\ref{cor.D,prop.W} compose to form a rig functor $\hh\coloneqq \R\circ \T$ given by
\begin{equation}\label{eqn.entropy_data}
\begin{aligned}
	\polycart&\To{\mathcal{h}}\smset\times\smset\op\\
	p&\mapsto (\dot{p}(1),\Gamma(\dot{p}\yon)).
\end{aligned}
\end{equation}

We refer to $\hh(p)\in\smset\times\smset\op$ as \emph{the categorical entropy} of the polynomial $p$. This pair of sets leaves behind any semblance of the probability distribution associated with $p$, but it retains the data necessary to compute $p$'s entropy---as we'll see in Theorem~\ref{thm.main}---and it is rig-functorial in $p$. 

\subsection{Shannon entropy}\label{sec.entropy}

Writing $\log$ to denote $\log_2$, we define a partial function $L\colon\ob(\smset\times\smset\op)\to\rr$ by
\begin{equation}\label{eqn.L}
(A,B)\mapsto \log A-\frac{\log B}{A}.
\end{equation}
Equivalently, $L(A,B)=\log A-\log\sqrt[A]{B}$. When $A=0$ and $B=1$, we define this function to be $L(0,1)\coloneqq 0$; for all cases where $A=0$, or $B=0$, or either $A$ or $B$ is infinite, we leave $L(A,B)$ undefined. We will be only interested in this map when it is composed with the categorical entropy $\mathcal{h}$ from \eqref{eqn.entropy_data}, and Lemma~\ref{lemma.dontworry} below says that we do not need to worry about the undefined cases.

\begin{lemma}\label{lemma.dontworry}
Let $p\in\polycart$ with categorical entropy $(A,B)\coloneqq\hh(p)$, and suppose that $\card\dot{p}(1)<\infty$. Then we have that
\begin{enumerate}[label=\roman*.]
	\item $B\neq 0$,
	\item if $A=0$ then $B=1$, and
	\item both $A$ and $B$ are finite.
\end{enumerate}
\end{lemma}
\begin{proof}
By definition of $\hh$, we have that $A\coloneqq\dot{p}(1)$ and $B\coloneqq\Gamma(\dot{p}\yon)$. 
\begin{enumerate}[label=\roman*.]
	\item One easily checks using \eqref{eqn.global} that for any $q\in\poly$, the set $\Gamma(q\yon)\neq0$ is nonempty since every $(q\yon)$-type has at least one term. 
	\item If $\dot{p}(1)=0$ then $p\in\smset$ is constant, so $\dot{p}\yon=0$ as well, and $\Gamma(0)=1$ by \eqref{eqn.global}.
	\item By assumption $\card A=\card\dot{p}(1)<\infty$. For $B$, note that there are only a finite number of $I\in p(1)$ for which $p[I]$ is nonempty, so by \eqref{eqn.global} and \eqref{eqn.dotp} the set $\Gamma(\dot{p}\yon)$ is finite.
\qedhere
\end{enumerate}
\end{proof}

\begin{remark}[Log aspect ratio]\label{rem.log_aspect_ratio}
With the interpretation of an object $(A,B)\in\smset\times\smset\op$ as a rectangle with length $A$ and width $\sqrt[A]{B}$, as in Remark~\ref{rem.geomean}, we can think of $L(A,B)=\log A -\log\sqrt[A]{B}$ as its \emph{log aspect ratio}, the log of its length divided by its width. This is a quantity that has come up in the study of vision \cites{talbot2011arc,dickinson2017separate}, though we're making no claim about whether this connection is meaningful.
\end{remark}

\begin{definition}[Empirical distribution]\label{def.empirical}
Let $p\neq 0$ be a nonzero polynomial and suppose that the cardinality of $\dot{p}(1)\in\smset$ is finite, $\card\dot{p}(1)<\infty$. We define the \emph{empirical distribution defined by $p$} to be the following function $P\colon p(1)\to[0,1]$:
\[
 P(I)\coloneqq\frac{\card p[I]}{\card\dot{p}(1)}
\]
Note that $1=\sum_{I\in p(1)}P(I)$, so $P$ is indeed a probability distribution.
\end{definition}

\begin{remark}\label{rem.poly_prod_monoidal}
One may ask how to view $\poly$'s monoidal structures, especially $+$ and $\otimes$, under the correspondence from Definition~\ref{def.empirical}. Suppose given polynomials $p, q\in\poly$ with associated probability distributions $P_p$ and $P_q$. For Dirichlet product we have
\[
P_{p\otimes q}=P_p\otimes P_q
\]
where the left-hand side is the probability distribution associated to $p\otimes q$ and the right-hand side is the usual tensor (independent) product of probability distributions. For sums we have
\[
P_{p+q}=\frac{\dot{p}(1)}{\dot{p}(1)+\dot{q}(1)}P_p + \frac{\dot{q}(1)}{\dot{p}(1)+\dot{q}(1)}P_q
\]
the convex combination of $P_p$ and $Q_q$, weighted according to the relative number of draws $\dot{p}(1)$ and $\dot{q}(1)$ in each.
\end{remark}

%
Recall that the Shannon entropy $H(P)$ of a probability distribution $P\colon X\to[0,1]$ is given by
\[
H(P)\coloneqq-\sum_{x\in X}P(x)\log P(x).
\]

The following theorem could be summarized as follows: ``thinking of $p\in\polycart$ as a statistical sample, the entropy $H(P)$ of the corresponding probability distribution $P$ is equal to the log ratio of the rectangular aspect of $p$'s total polynomial''; see Remarks~\ref{rem.total},~\ref{rem.geomean},~and~\ref{rem.log_aspect_ratio}.

\begin{theorem}\label{thm.main}
Let $p\neq0$ be a nonzero polynomial with $\card\dot{p}(1)<\infty$, and let $P$ be the empirical distribution defined by $p$. Then the following equation holds
\[
H(P)=L(\hh(p))
\]
where $H$ is the Shannon entropy and $L,\hh$ are as defined in Eqs.~\eqref{eqn.entropy_data}~and~\eqref{eqn.L}.
\end{theorem}
\begin{proof}
We need to show that the following holds:
\[
H(P)=\log\dot{p}(1)-\frac{\log\Gamma(\dot{p}\yon)}{\dot{p}(1)}.
\]
With the fact $\Gamma(\dot{p}\yon)\cong\prod_{I\in p(1)}p[I]^{p[I]}$ from Example~\ref{ex.global_D} in hand, this is a routine calculation:
\begin{align*}
	H(P)&\coloneqq
	-\sum_{I\in p(1)}\frac{\card p[I]}{\card\dot{p}(1)}\log \frac{\card p[I]}{\card\dot{p}(1)}\\&=
	\frac{1}{\card\dot{p}(1)}\sum_{I\in p(1)}\card p[I]\big(\log\card\dot{p}(1)-\log\card p[I]\big)\\&=
	\frac{1}{\card\dot{p}(1)}\left(\card\dot{p}(1)\log\card\dot{p}(1)-\log\prod_{I\in p(1)}\card p[I]^{\card p[I]}\right)\\&=
	\log\card\dot{p}(1)-\frac{\log\Gamma(\dot{p}\yon)}{\dot{p}(1)}
\end{align*}
\end{proof}

\begin{example}[Entropy of a uniform distribution]
It is well-known and easy to calculate that if $P$ is a uniform distribution on $A$ elements, then $H(P)=\log(A)$. There are many samples that correspond to $P$; what differs are their sample sizes. The sample in which $AB$-many observations are taken---each outcome occurring $B$-many times---corresponds to the polynomial $A\yon^B$. 

Our formula for entropy needs to agree, and it does. The rectangular aspect of the total polynomial is $\hh(p)\cong(AB,B^{AB})$: length $AB$ and width $B=\sqrt[AB]{B^{AB}}$, so its log aspect ratio is
\[
L(\hh(p))=\log(AB)-\frac{\log(B^{AB})}{AB}=\log A.
\qedhere
\]
\end{example}

\printbibliography 
\end{document}